\documentclass{article}

\usepackage{amsmath,amssymb}
\usepackage{amsthm}
\usepackage{enumitem}
\usepackage[toc]{appendix}
\usepackage{multicol}
\usepackage{hyperref}
\usepackage{url}
\usepackage{graphicx}

\usepackage{color}

\usepackage{listings}
\usepackage{tikz}
\tikzset{
every node/.style={circle, inner sep=2pt}
}
\usetikzlibrary{matrix,arrows,calc}

\newtheorem{theorem}{Theorem}
\newtheorem{lemma}[theorem]{Lemma}
\newtheorem{proposition}[theorem]{Proposition}
\newtheorem{corollary}[theorem]{Corollary}

\theoremstyle{definition}
\newtheorem{definition}[theorem]{Definition}
\newtheorem{observation}[theorem]{Observation}
\newtheorem{remark}[theorem]{Remark}
\newtheorem{example}[theorem]{Example}
\newtheorem{problem}{Problem}
\newtheorem{question}[theorem]{Question}

\newenvironment{thm}{\begin{theorem}}{\end{theorem}}
\newenvironment{lem}{\begin{lemma}}{\end{lemma}}

\renewcommand{\>}{\rangle}
\newcommand{\<}{\langle}

\title{Spectral upper bound on the quantum $k$-independence number of a graph}

\author{Pawel Wocjan\thanks{\texttt{Pawel.Wocjan@ibm.com}, IBM Quantum, IBM T.J. Watson Research Center, Yorktown Heights, NY 10598, USA}  \quad Clive Elphick\thanks{\texttt{clive.elphick@gmail.com}, School of Mathematics, University of Birmingham, Birmingham, UK} \quad Aida Abiad\thanks{\texttt{a.abiad.monge@tue.nl},  Department of Mathematics and Computer Science, Eindhoven University of Technology, The Netherlands, Department of Mathematics: Analysis, Logic and Discrete Mathematics, Ghent University, Belgium, Department of Mathematics and Data Science of Vrije Universiteit Brussel, Belgium}}

\date{}

\begin{document}

\maketitle
\abstract{A  well known upper bound for the independence number $\alpha(G)$ of  a graph $G$, due to Cvetkovi\'{c}, is that 
\begin{equation*}
\alpha(G) \le n^0 + \min\{n^+ , n^-\}
\end{equation*}
where $(n^+, n^0, n^-)$ is the inertia of $G$. We prove that this bound is also an upper bound for the quantum independence number $\alpha_q$(G), where $\alpha_q(G) \ge \alpha(G)$ and for some graphs $\alpha_q(G) \gg \alpha(G)$. We identify numerous graphs for which $\alpha(G) = \alpha_q(G)$, thus increasing the number of graphs for which $\alpha_q$ is known. We also demonstrate that there are graphs for which the above bound is not exact with any Hermitian weight matrix, for $\alpha(G)$ and $\alpha_q(G)$.   Finally, we show this result in the more general context of spectral bounds for the quantum $k$-independence number, where the $k$-independence number is the maximum size of a set of vertices at pairwise distance greater than $k$.}

\section{Introduction and Motivation}

Elphick and Wocjan \cite{elphick19} proved that many spectral lower bounds for the chromatic number, $\chi(G)$, are also lower bounds for the quantum chromatic number, $\chi_q(G)$. This was achieved using the linear algebra tools of pinching and twirling and a combinatorial definition of $\chi_q(G)$ due to Man\u{c}inska and Roberson \cite{mancinska16}. In a different paper Man\u{c}inska and Roberson \cite{mancinska162} defined a the quantum independence number $\alpha_q(G)$, using quantum homomorphisms, where $\alpha_q(G) \ge \alpha(G).$ Analogously to $\chi_q(G)$, the \emph{quantum independence number} $\alpha_q(G)$ is the maximum integer $t$ for which two players sharing an entangled quantum state can convince an interrogator that the graph $G$ has an independent set of size $t$. There exist graphs $G$ for which there is an exponential separation between the independence number $\alpha(G)$ and $\alpha_q(G)$ \cite{mancinska162}.

The subject of quantum graph parameters has been extensively studied in the past decade, due to its connections to a number of subjects, including quantum information theory, operator theory, combinatorics and optimisation. 
The motivation for studying quantum graph parameters is described, for example, by Cameron \emph{et al.} \cite{cameron07} and in  \cite{mancinska162} and \cite{elphick19}. We add upon this subject by defining an extension of the usual quantum independence number of a graph, called the quantum $k$-independence number, which is also motivated by its classical counterpart, the $k$-independence number of a graph. The $k$-\emph{independence number of a graph}, $\alpha_k(G)$, is the maximum size of a set of vertices at pairwise distance greater than $k$. Upper bounds on this graph parameter appeared in \cite{fh1997,f1999,abiad16,AGF2019,OShiTaoqiu2019, ACFNS2020}. Note that $\alpha_1(G) = \alpha(G)$.  The quantum $k$-independence number can be regarded as a generalisation of the (classical) $k$-independence number. As set out in Definition~\ref{def:quantumkindependencenumber}, $\alpha_{kq}(G)$ is defined using $d$-dimensional orthogonal projectors, and $\alpha_k(G)$ corresponds to $d = 1$. It is also worth mentioning that in quantum information theory, $(\alpha_{kq}(G) - \alpha_k(G))$ is a measure of the benefit of quantum entanglement.

 In this article we prove an inertial spectral upper bound for the quantum counterpart of the $k$-independence number, $\alpha_{kq}$, which is the maximum integer $t$ for which two players sharing an entangled quantum state can convince an interrogator that the graph $G$ has a $k$-independent set of size $t$. Recall that it is not known whether quantum counterparts of $\alpha$ or $\chi$ are computable functions \cite{MR}. In this work, we identify numerous graphs for which $\alpha(G) = \alpha_q(G)$, hence increasing the number of graphs for which $\alpha_q$ is computed. Also, as a corollary of our main result (the inertia bound is also valid to upper bound the quantum $k$-independence number), one can use the methods proposed in \cite{ACFNS2020} to optimize our inertial bound and find exact values for $\alpha_{kq}$.

\section{Definitions and Notation}

The quantum independence number was originally defined using quantum homomorphisms and can also be defined using nonlocal games. In this work we require a combinatorial definition of $\alpha_q$, such as the one which appears in \cite{laurent15} (see Definition 2.8). This is generalised to $\alpha_{kq}$ as follows, where for matrices $X,Y\in \mathbb{C}^{d\times d}$, the trace inner product (also called Hilbert-Schmidt inner product) is defined as
\begin{equation*}
\<X, Y\>_{\mathrm{tr}} = \mathrm{tr}(X^\dagger Y)\,.
\end{equation*}

Clearly $k = 1$ corresponds to the quantum independence number $\alpha_q(G)$. 

\begin{definition}\label{def:quantumkindependencenumber} The \emph{quantum $k$-independence number} of a graph $G = (V,E)$, denoted by $\alpha_{kq}$, is the maximum
integer $t$ for which there exist orthogonal projectors $P^{(u,i)} \in \mathbb{C}^{d\times d}$ for $u\in V(G)$, $i\in [t]$
satisfying the following conditions:


\begin{equation} \label{cond1}
\displaystyle\sum_{u\in V}P^{(u,i)}=I_d \text{ for all }i\in [t],
\end{equation}

\begin{equation} \label{cond2}
\langle P^{(u,i)},P^{(u,j)}\rangle_{\text{tr}}=0\quad \text{ for all }i\neq j\in [t], \text{ for all }u\in V(G),
\end{equation}

\begin{equation} \label{cond3}
\begin{split}
\langle P^{(u,i)},P^{(v,j)}\rangle_{\text{tr}}=0\quad &\text{ for all }i\neq j\in [t], \text{ for all }u,v\in V(G) \\
 & \text{ with dist}(u,v)\leq k.
\end{split}
\end{equation}

\end{definition}

Definition \ref{def:quantumkindependencenumber} can also be written in terms of nonlocal games. We adopt the notation from \cite{mancinska161}. A \emph{nonlocal game} is specified by four finite sets $A,B,Q,R$, a probability distribution $\pi$ on $Q\times R$ and a Boolean predicate $V: A\times B\times Q \times R \longrightarrow \{0,1\}$. The game proceeds as follows: Using  $\pi$ the interrogator samples a pair $(q,r) \in Q\times R$ and sends $q$ to Alice and $r$ to Bob. Upon receiving their questions the players respond
with $a\in A$ and $b\in B$, respectively. The players have knowledge of the distribution  $\pi$ and the predicate $V$ and
can agree on a common strategy before the start of the game, but they are not allowed to communicate after they receive their questions.

In the $(G, t,k)$-independent set game the players aim to convince an
interrogator that a graph $G$ contains a $k$-independent set of size $t$ (i. e., a set of $t$ vertices at distance greater than $k$ from each other). To
play the game the interrogator selects uniformly at random a pair of indices $(i,j) \in [t]\times [t]$ and sends $i$ to Alice
and $j$ to Bob. The players respond with vertices $u,v\in V(G)$ respectively. In order to win, the players need
to respond with the same vertex of $G$ whenever they receive the same index. Furthermore, if they receive
$i\neq j \in [t]$ they need to respond with two distinct vertices of $G$ at distance greater than $k$ from each other. 
 
 We can simplify the proof of our upper bound for $\alpha_{kq}(G)$ by defining a k-projective packing number of a graph, denoted by $\alpha_{kp}(G)$. We will then prove that $\alpha_{kq}(G) \le \alpha_{kp}(G)$, and that our spectral bound is an upper bound for $\alpha_{kp}(G).$

\begin{definition}\label{def:quantumkprojectivepacking}
A \emph{$d$-dimensional $k$-projective packing} of a graph $G = (V,E)$, denoted by $\alpha_{kp}$, is a collection of
orthogonal projectors $P^{(u)} \in \mathbb{C}^{d\times d}$ such that

\begin{equation}\label{eq:orthogonality}
    \langle P^{(u)},P^{(v)}\rangle_{\text{tr}}=0
\end{equation}

for all $u,v\in V(G)$ at distance at most $k$. The value of a projective packing using projectors
$P^{(u)} \in \mathbb{C}^{d\times d}$ is defined as 

$$\frac{1}{d}\displaystyle\sum_{v\in V}r^{(u)},$$
where $r^{(u)}$ denote the ranks of the projectors $P^{(u)}$. The k-projective packing number $\alpha_{kp}(G)$ of a
graph $G$ is defined as the supremum of the values over all projective packings of the
graph $G$. If $k = 1$, then $\alpha_{kp}(G) = \alpha_p(G)$, which is the projective packing number of $G$.
\end{definition}

In order to prove our main result we need the following two lemmas.

\begin{lemma}\label{lem:quantumkindepnumberquantumkdimprojectivepacking}
$\alpha_{kq}(G)\leq \alpha_{kp}(G)$.
\end{lemma}

\begin{proof}
It is known that the $k$-independence number of a graph $G$ is precisely the independence number of the graph formed by making all pairs of vertices of $G$ at distance at most $k$ adjacent. Denote this latter graph by $G^{[k]}$. Then $\alpha_{kq}(G)=\alpha_q(G^{[k]})$ and $\alpha_{kp}=\alpha_{p}(G^{[k]})$. Using that $\alpha_q\leq \alpha_p$ \cite{mancinska161}, it holds that $\alpha_{kq}(G)\leq \alpha_{kp}(G)$. 
\end{proof}

We will use the following result to reformulate the conditions on the orthogonal projectors of a k-projective packing as conditions on their eigenvectors.  We omit the proof of this basic result.
\begin{lem}\label{lem:elementary}
Let $P, Q\in\mathbb{C}^{d\times d}$ be two arbitrary orthogonal projectors of rank $r$ and $s$, respectively. Let
\[
P = \sum_{k\in[r]} |\psi_k\>\<\psi_k| \quad\mbox{and}\quad Q = \sum_{\ell\in[s]} |\phi_\ell\>\<\phi_\ell|
\]
denote their spectral resolutions, respectively.  Then, the following two conditions are equivalent:
\begin{align}
\< P, Q\>_{\mathrm{tr}} = 0 & \\
\<\psi_k | \phi_\ell \> = 0 & \quad \mbox{for all}\quad k\in[r], \ell\in[s]\,.
\end{align}
\end{lem}

\section{Spectral bound for \texorpdfstring{$\alpha_{kq}(G)$}{alphakq(G)} and \texorpdfstring{$\alpha_q(G)$}{alphaq(G)}}

Let $A$ denote the adjacency matrix of $G$ and let $p_k(A)\in\mathbb{R}_k[x]$ denote a polynomial function of $A$ of degree at most $k$. Let
\begin{align*}
W(p_k) &= \max_{u \in V} \big\{ p_k(A)_{uu} \big\}, \\
w(p_k) &= \min_{u \in V} \big\{p_k(A)_{uu} \big\}.
\end{align*}
If $p_k(A) = A^k$ then $W(p_k)$ is the maximum number of closed walks of length $k$ where the maximum is taken over all vertices and $w(p_k) $ is the minimum number of closed walks of length $k$.

Abiad \emph{et al.} \cite{AGF2019} proved the following result.

\begin{theorem}\cite{AGF2019}
Let $\lambda_1 \ge ... \ge \lambda_n$ denote the eigenvalues of the adjacency matrix $A$ of a graph $G$, and let $p_k\in \mathbb{R}_k[x]$. Then, 
\begin{equation*}
\alpha_k(G) \le \min\big\{ \big| i : p_k(\lambda_i) \ge w(p_k) \big|,\, \big| i : p_k(\lambda_i) \le W(p_k) \big| \big\}.
\end{equation*}
\end{theorem}

If we let $k = 1$ and $p_k(A) = A$, this reduces to the well known inertia bound due to Cvetkovi\'{c} \cite{cvetkovic73}:

\begin{theorem} \cite{cvetkovic73} The independence number of a graph is bounded from above by
\begin{equation*}
\alpha(G) \le n^0(A) + \min\{n^+(A) , n^-(A)\},
\end{equation*}
where $(n^0, n^-, n^+)$ are the numbers of zero, negative and positive eigenvalues of $A$.
\end{theorem}

Our principal result is as follows.

\begin{thm}\label{thm:main}
Let $\lambda_1 \ge \ldots \ge \lambda_n$ denote the eigenvalues of the adjacency matrix $A$ of a graph $G$, and let $p_k\in \mathbb{R}_k[x]$. Then 
\begin{equation*}
\alpha_{kq}(G) \le \alpha_{kp}(G) \le 
\min\big\{ \big| i : p_k(\lambda_i) \ge w(p_k) \big|,\,  \big|i : p_k(\lambda_i) \le W(p_k) \big| \big\}.
\end{equation*}

\end{thm}
\begin{proof}

Let $G^{[k]}$ denote the graph defined as in Lemma~\ref{lem:elementary} and $A^{[k]}$ denote its adjacency matrix. 
Let $A^\ell$ denote the $\ell$th power of the adjacency matrix $A$ where $\ell\le k$. Let $u,v\in V$ be two arbitrary but different vertices of $V$.  It is clear that $(A^{[k]})_{uv}=0$ implies $(A^\ell)_{uv}=0$.  Consequently $(p_k(A))_{uv}=0$ for all $p_k\in \mathbb{R}_k[x]$ since $p_k(A)$ is a linear combination of the powers $A^0,\ldots,A^k$.

Using Lemma~\ref{lem:elementary} it is possible to obtain an equivalent formulation of a $k$-projective packing in terms of vectors instead of projectors.  We can equivalently consider the orthonormal vectors
\begin{equation*}
    |\Psi^{(u,i)}\> = |u\> \otimes |\psi^{(u,i)}\>
\end{equation*}
where the spectral resolutions of the projectors $P^{(u)}$ are given by
\begin{equation*}
    P^{(u)} = \sum_{i=1}^{r^{(u)}} |\psi^{(u,i)}\>\<\psi^{(u,i)}|.
\end{equation*}
The orthogonality condition (\ref{eq:orthogonality}) then directly translates to
\begin{equation*}
    \<\Psi^{(u,i)} | \left( A^{[k]}\otimes I_d \right) |\Psi^{(v,j)}\> = 
    \delta_{uv} \cdot \delta_{ij} \cdot (A^{[k]})_{uv}.
\end{equation*}
Now let $S$ denote the matrix whose columns are $|\Psi^{(u,i)}\>$ for $u\in V$ and $i\in [r^{(u)}]$.
We obtain
\begin{equation*}
    S^\dagger ( p_k(A)\otimes I_d ) S = \mathrm{diag} \left( d^{(u,i)} : u\in V, i\in [r^{(u)}] \right)
\end{equation*}
where the diagonal entries are given by
\begin{equation*}
    d^{(u,i)} = (p_k(A))_{uu}.
\end{equation*}
Let $w(p_k)$ denote the minimum of $(p_k(A))_{uu}$ for $u\in V$. Using interlacing, it now follows that there must be at least $r$ eigenvalues of $p_k(A)\otimes I_d$ that are larger than the smallest eigenvalue of $D=\mathrm{diag}(d^{(u,i)})$.  The latter is equal to $w(p_k)$ by definition.  Equivalently, there must be at least
$r/d$ eigenvalues of $p_k(A)$ that are at least $w(p_k)$.

This yields the first upper bound
\begin{equation*}
    \alpha_{kp} \le \big| \big\{ j : p_k(\lambda_j) \ge w(p_k) \big\} \big|.
\end{equation*}

Let $W(p_k)$ denote the maximum of $(p_k(A))_{uu}$ for $u\in V$. The second upper bound
\begin{equation*}
    \alpha_{kp} \le \big| \left\{ j : p_k(\lambda_j) \le W(p_k) \right\} \big|
\end{equation*}
is proved analogously. 

\end{proof}

Letting $p_1(A) = A$ we immediately obtain that:

\begin{corollary}\label{cor:one}
\begin{equation*}\label{eq:quantum}
\alpha(G) \le \alpha_q(G) \le \alpha_p(G) \le n^0(A) + \min\{n^+(A) , n^-(A)\}.
\end{equation*}
\end{corollary}
Abiad \emph{et al.} \cite{abiad16} proved that, when we restrict to the case of $p_k(A)=A^k$, the spectral bound from Theorem \ref{thm:main} is tight for a certain infinite family of graphs (see section 3.1 in \cite{abiad16}).  Since the quantum $k$-independence number is sandwiched between the classical $k$-independence number and our spectral bound, we can say that for that family of graphs the above bound for $\alpha_{kq}$ is also tight and the classical and quantum parameters $\alpha_{k}$ and $\alpha_{kq}$ coincide.

Note that Theorem~\ref{thm:main} and therefore Corollary~\ref{cor:one} are valid for weighted adjacency matrices of the form $H \circ A$, where $H$ is an arbitrary Hermitian matrix and $\circ$ denotes the Hadamard product (also called the Schur product). This is because:

\[
(A^p)_{uv} = 0 \Rightarrow ((H \circ A)^p)_{uv} = 0 \mbox{   and   } (p(A))_{uv} = 0 \Rightarrow (p(H \circ A))_{uv} = 0.
\]

\section{Alternative upper bounds for \texorpdfstring{$\alpha_q(G)$}{alphaq}}
It is known (see for example Section 6.18 of \cite{roberson13}) that:
\begin{align*}
& \alpha(G) \le \alpha_q(G) \le \lfloor \vartheta'(G) \rfloor \le \vartheta'(G) \le \vartheta(G) \le \vartheta^+(G) \le \lceil \vartheta^+(G) \rceil \le \\
& \chi_q(\overline{G}) \le \chi(\overline{G}) \,,
\end{align*}
where $\vartheta'$, $\vartheta$, and $\vartheta^+$ are the Schrijver, Lov\'asz and Szegedy theta functions. 

Hoffman, in an unpublished paper, proved that for $\Delta$-regular\footnote{We use the unconventional symbol $\Delta$ instead of $d$ for the degree of regular graphs because $d$ is the dimension of the Hilbert space used in the definition of the quantum independence number.} graphs:
\begin{equation*}
\alpha(G) \le \frac {n|\lambda_n|}{\Delta + |\lambda_n|},
\end{equation*}
where $\lambda_n$ is the smallest eigenvalue of the adjacency matrix $A$. This result is typically proved using interlacing of the quotient matrix, and is known as the Hoffman bound or ratio bound. 

Lov\'asz \cite[Theorem 9]{lovasz79} proved that for $\Delta$-regular graphs:
\begin{equation*}\label{eq:thetaHoffman}
\vartheta(G) \le \frac {n|\lambda_n|}{\Delta + |\lambda_n|}.
\end{equation*}
It is therefore immediate that the Hoffman bound is an upper bound for $\alpha_q(G)$ for regular graphs.

Van Dam and Haemers \cite{dam98} proved that for \emph{any} graph 
\begin{equation*}
\alpha(G) \le \frac{n(\mu_1 - \delta)}{\mu_1}\,,
\end{equation*}
where $\delta$ is the minimum degree of $G$ and $\mu_1$ is the largest eigenvalue of the Laplacian matrix of $G$. This bound equals the Hoffman bound for regular graphs. 

Bachoc \emph{et al.} subsequently proved (see section~\ref{sec:implications} in \cite{bachoc17}), in the context of simplicial complexes, that for any graph:

\begin{equation*}
\vartheta(G) \le \frac{n(\mu_1 - \delta)}{\mu_1}.
\end{equation*}

It is therefore immediate that the van Dam and Haemers bound is an upper bound for $\alpha_q(G)$.

\section{Implications for \texorpdfstring{$\alpha_q(G)$}{alphaq}  and for \texorpdfstring{$\alpha(G)$}{alpha}}\label{sec:implications}

It follows from Theorem~\ref{thm:main} that any graph with $\alpha(G) = n^0 + \min{(n^+ , n^-)}$, has $\alpha_q = \alpha$. This is the case for numerous graphs, including odd cycles, perfect, folded cubes, Kneser, Andrasfai, Petersen, Desargues, Grotzsch, Heawood, Clebsch and Higman-Sims graphs. Furthermore if the inertia bound is tight with an appropriately chosen weight matrix then again $\alpha_q = \alpha$. This is the case for all bipartite graphs. There are also many graphs, including Chvatal, Hoffman-Singleton, Flower Snark, Dodecahedron, Frucht, Octahedron, Thomsen, Pappus, Gray, Coxeter and Folkman   for which $\alpha = \lfloor\vartheta\rfloor$, so again $\alpha_q = \alpha$. For all such graphs there are no benefits from  quantum entanglement for independence. The Clebsch graph demonstrates that the inertia bound is not an upper bound for $\lfloor \vartheta'(G) \rfloor$. 

Elzinga and Gregory \cite{elzinga10} asked whether there exists a real symmetric weight matrix $W$ for every graph $G$ such that:

\begin{equation}\label{eq:elzinga}
\alpha(G) = n^0(W) + \min{(n^+(W) , n^-(W))}?
\end{equation}

They demonstrated experimentally that this is true for all graphs with up to 10 vertices, and for vertex transitive graphs with up to 12 vertices. Sinkovic \cite{sinkovic18} subsequently proved that there is no real  symmetric weight matrix for which~(\ref{eq:elzinga}) is tight for Paley 17.  This leaves open, however,  whether there is always a Hermitian weight matrix for which~(\ref{eq:elzinga}) is exact. 

It follows from Theorem~\ref{thm:main}, that every graph with $\alpha < \alpha_q$ is a counter-example to~(\ref{eq:elzinga}) for real symmetric and Hermitian weight matrices. This leads to the question of whether ~(\ref{eq:elzinga}) is true for $\alpha_q$ or $\alpha_p$. It follows from Theorem~\ref{thm:main} that the answer is no, because for some graphs, such as the line graph of the cartesian product of $K_3$ with itself, the projective packing number is non-integral.

There are also numerous regular graphs for which the Hoffman bound on $\alpha(G)$ is exact, but the unweighted inertia bound is not. Examples include the Shrikhander, Tesseract, Hoffman and Cuboctahedral graphs. There are also many regular graphs where the floor of the Hoffman bound is exact, but the unweighted inertia bound is not. Examples include some circulant, cubic and quartic graphs. For all of these graphs, $\alpha_q = \alpha$.

Appendix A in \cite{abiad16} demonstrates, however,  that it is hard to find well known graphs for which there is equality in Theorem~\ref{thm:main} when $k \ge 2$ and $p_k(A)=A^k$. 

It would be interesting to find the graph with the smallest number of vertices that has $\alpha(G) < \alpha_q(G)$. Such a graph must have at least 11 vertices (given the experimental results due to Elzinga and Gregory).  The smallest such graph that we know of is due to Piovesan (see Figure 3.1 in \cite{piovesan16}) which has 24 vertices, with $\chi = \alpha =5, \chi_q = 4$ and $\alpha_q \ge 6$.

\section{Conclusion}

To conclude we illustrate the differences between classical and quantum graph parameters, by summarising results in \cite{mancinska162} for orthogonality graphs.  The orthogonality graph $\Omega(n)$ has vertex set the set of $\pm1$-vectors of length $n$, with two vertices adjacent if they are orthogonal. With $n$ a multiple of 4, $\chi_q(\Omega(n)) = n$ but for large enough $n$, $\chi(\Omega(n))$ is exponential in $n$.  Similarly $\alpha_q(\Omega(n) \square K_n) = 2^n$ but $\alpha(\Omega(n) \square K_n) \le n(2 - \epsilon)^n$, for some $\epsilon > 0$, where $\square$ denotes the Cartesian product.

Therefore, the spectral bounds in this paper and in \cite{elphick19} for quantum graph parameters demonstrate the weaknesses of such bounds for classical graph parameters for some families of graphs.

\subsection*{Acknowledgments}

We would like to thank David Roberson for helpful comments on an earlier version of this paper, in particular in regard to the projective packing number. We would also like to thank David Anekstein for testing various ideas for this paper experimentally. Aida Abiad would like to thank 
the  organizers  and the  participants  of  the  workshop Analytical and combinatorial aspects of quantum information theory, in particular,  Gabriel Coutinho and Laura Man\u{c}inska, for helpful discussions on the quantum definition of the $k$-independence number.

This research has been partially supported by National Science Foundation Award 1525943.


\end{document}